\documentclass[12pt, 14paper,reqno]{amsart}
\usepackage{amsmath,amsfonts,amssymb}

\usepackage{longtable}
\usepackage{graphics}
\usepackage{amsfonts,amssymb,color}
\usepackage[mathscr]{eucal}
\usepackage{amsmath, amsthm}
\usepackage{mathrsfs}
\usepackage{amsbsy}
\usepackage{wasysym}
\usepackage{mathtools}
\input xypic
\xyoption{all}
\theoremstyle{plain}
\usepackage{color}
\newtheorem{theorem}{Theorem}

\newtheorem{corollary}{Corollary}[section]
\newtheorem{proposition}{Proposition}[section]
\theoremstyle{proof}
\theoremstyle{definition}

\newcommand \ZZ {{\mathbb Z}}

\newcommand \CC {{\mathbb C}}
\newcommand \PR {{\mathbb P}}
\newcommand \AF {{\mathbb A}}

\newcommand \QQ {{\mathbb Q}}

\newcommand \bcL {{\mathscr L}}

\newcommand \bcP {{\mathscr P}}

\newcommand \bcS {{\mathscr S}}
\newcommand \bcV {{\mathscr V}}

\newcommand \bcZ {{\mathscr Z}}

\newcommand \Spec {{\rm {Spec}}}

\newcommand \Supp{{\rm {supp}}}
\newcommand \Pic {{\rm {Pic}}}

\newcommand \ev {{\it {ev}}}

\newcommand \Hom {{\rm Hom}}

\newcommand \Cl{{\rm{Cl}}}
\newcommand \Proj {{\rm {Proj}}}

\newcommand \CH {{\it {CH}}}
\newcommand \wt {\widetilde }

\theoremstyle{remark}
\newtheorem{remark}{Remark}[section]

\theoremstyle{lamma}

\newtheorem{thm}{Theorem}[section]

\newtheorem{prop}[thm]{Proposition}
\newtheorem{thma}{Theorem}

\numberwithin{equation}{section}
\numberwithin{lemma}{section}
\numberwithin{theorem}{section}

\usepackage{bigints}
\usepackage{amsmath}
\usepackage{amsfonts}
\usepackage{amssymb}
\usepackage{amssymb, amsmath, amsthm}
\usepackage[breaklinks]{hyperref}
\theoremstyle{thmrm}

\usepackage{graphicx}
\begin{document}
\title[Chow groups, pull back and class groups]{Chow groups, pull back and class groups}
\author{\small Kalyan Banerjee and Azizul Hoque}
\address{Kalyan Banerjee @SRM University AP, Mangalagiri-Mandal, Amaravati, Guntur-522240, Andhra Pradesh, India.}
\email{kalyan.ba@srmap.edu.in}
\address{Azizul Hoque @Department of Mathematics, Faculty of Science, Rangapara College, Rangapara, Sonitpur-784505, Assam, India.}
\email{ahoque.ms@gmail.com}

\maketitle
\begin{abstract}
Let $S$ be a certain affine algebraic surface over $\QQ$ such that it admits a regular map to $\AF^2/\QQ$. We show that any non-trivial torsion element in the  Chow group $\CH^1(S)$ can be pulled back to ideal classes of quadratic fields whose order can be made as large as possible. This gives an affirmative answer to a question analogous to one raised by Agboola and Pappas, in the case of certain affine algebraic surfaces. Spreading out $S$ over $\ZZ$ and for a closed point $P\in \AF^2/\ZZ$, we show that the cardinality of a subgroup of the Picard group of the fiber $S_P$ remains unchanged when $P$ varies over a Zariski open subset in $\AF^2/\ZZ$. We also show by constructing an element of odd order $n\geq 3$ in the class group of certain imaginary quadratic fields that the Picard group of $S_P$ has a subgroup isomorphic to $\ZZ/n\ZZ$.
\end{abstract}

\section{Introduction}
Notations: By $\AF^2$ we mean an affine plane; let us  denote a spread (or an integral model) of $S$ over $\ZZ$ by $\bcS$. Given a point $P$ in $\AF^2/\ZZ$, by $\bcS_P$ we mean a fiber of $\bcS\to \AF^2/\ZZ$ at $P$; we denote by $K$, a number field. For a number field $K$, by $\mathcal{O}_K$ we mean the ring of integers of $K$.

Let $S$ be an affine algebraic surface defined over $\QQ$ by the equation
$$y^2=t^2q^2-z^n,$$
where $y,t,z$ are indeterminates and $n, q\in\ZZ$ with $q$ prime.
Suppose that it admits the  regular map to $\AF^2$ over $\QQ$ given by
$$(y,t,z)\mapsto (t,z).$$
We can spread the surface over $\ZZ$ and then consider the spread over $\AF^2/\ZZ$. Denote this spread by $\bcS$. Given a point $P$ in $\PR^2/K$, it corresponds to a point on $\PR^2/\mathcal{O}_K$ and vice versa. One of the goals of this paper is to understand  the following situation in the case of imaginary quadratic fields. Let $\bcL$ be a non-trivial torsion Weil divisor class in the  Chow group $\CH^1(\bcS)$ of co-dimension one cycles on $\bcS$.
For $P \in \AF^2/\ZZ$, consider the pull back of $\bcL$ over the normalization of $\bcS_P$, and denote it by $\wt{\bcS_P}$. This corresponds to an element in $\CH^1(\wt{\bcS_P})$ which is isomorphic to $\Cl(\QQ(S_P))$. Here $\QQ(S_P)$ is the field given by
$$\QQ[y]/(y^2-m^2q^2+l^n)$$
for a point $(m,l)$ in $\AF^2/\ZZ$, such that
$y^2-m^2q^2+l^n$
is irreducible.

The question is whether this element is a non-trivial ideal class in the number field $\QQ(S_P)$?

This is an analogous question raised by Agboola and Pappas in \cite{AP}, that is: let $S$ be any affine, smooth algebraic surface fibered over $\AF^2$, then any non-trivial element in the Picard group of $\bcS$ specializes to a non-trivial element of the ideal class group of the number field $\QQ(S_P)$ for a general choice of $P$. Gillibert and Levin \cite{GL} affirmatively answered the question raised by Agboola and Pappas in the case of torsion line bundles on a hyperelliptic curve. Gillibert \cite{GI} also considered a similar question in case of non-trivial line bundles of degree zero on certain hyperelliptic curves which are not torsion.

Another goal of this paper is to show, for a closed point $P$ in $\AF^2/\QQ$ (which corresponds to a point in $\AF^2/\ZZ$), 
the cardinality of certain  subgroups of the divisor class group of the normalization of the fiber $\bcS_P$ remains constant when $P$ varies over a Zariski open subset in $\AF^2/\ZZ$. Moreover, we show that this class group has a subgroup isomorphic to $\ZZ/n\ZZ$ by constructing an element of odd order $n\geq 3$ in the class group of certain families of imaginary quadratic fields. Producing an element of a given order in a class group of an imaginary quadratic field (in fact, any number field) provides more information about the class group than showing the divisibility of the class number of the same field. However, the divisibility properties of the class number of a number field help us to understand the structure of class group. Many families of imaginary quadratic fields with class number divisible by a given integer $n\geq 2$ are known. Most of such families are of the type $\mathbb{Q}(\sqrt{k^2-\lambda t^n})$, where $k$ and $t$ are positive integers with some restrictions, and $\lambda =1, 2, 4$ (for $\lambda=1$ see
\cite{AC55, GL-18, H-22, KI09,  SO00}; for $\lambda=2$ see \cite{CH-19, CH-23, HC-18}
and for $\lambda=4$ see \cite{GR01, HO20, LO09}).
Here, we consider the family $K_{k, p, n}=\mathbb{Q}(\sqrt{k^2-p^n})$, where $k, n$ are positive integers and $p$ is an odd prime such that $k^2<p^n$. We produce an ideal class of order $n$ in $K_{k,p,n}$, and then we use this ideal class to show that the divisor class group of the normalization of the fiber $\bcS_P$ has a subgroup which is isomorphic to $\ZZ/n\ZZ$. The main idea that used to connect the two concepts of $n$-divisibility is to understand the fiber of the natural map from the relative Chow schemes to the relative Picard scheme associated to a flat family of projective schemes.

\section{Mumford-Ro\u{i}tman argument on Chow schemes and relative Picard schemes}

Let $B$ be a smooth scheme over $\Spec(\ZZ)$. For a flat morphism $X\to B$ of projective schemes, consider a vector bundle $E$ of constant rank over $B$ such that it gives a closed embedding of $X/B$ into some projective space $\PR(E)/B$, where
$\PR(E)=\Proj_B(E)$.
This determines the degree of the subschemes of relative co-dimension one subschemes in $X/B$. We consider the Chow variety\footnote{We refer \cite{Ko} for the theory of Chow variety.}, $C^1_d(X/B)$ of relative co-dimension one subschemes of $X\to B$ of degree $d$, that is,
$$C^1_{d}(X/B)=\{(D_b,b)|\Supp(D_b)\subset X_b, \deg(D_b)=d\}.$$
Let us explain what is $D_b$. Consider the closed embedding of $X/B$ into $\PR(E)$. Let us consider a closed subscheme $D/B$ of relative co-dimension one in $\PR(E)$ and consider its intersection with $X/B$, which is defined by $D_X:=D\times_B X$. This is well defined by the smoothness of $B$. In this terms, $D_b:=D_X\cap X_b$.

Then there is a natural map $C^1_d(X_b)\to \CH^1(X_b)$ associated to a $D_b$ to its divisor class $[D_b]$. In this set up, we consider the following:
$$\bcZ:=\{(b,D_b)|[D_b]=0 \in \CH^1(X_b)\}.$$
The proof of the following theorem is based on the idea introduced by Mumford in \cite{M}, which was further elaborated by Ro\u{i}tman in \cite{R} and Voisin in \cite{Voi}. This idea was used and presented in \cite{BG}.

\begin{theorem}
\label{theorem1}
$\bcZ$ is a countable union of Zariski closed subsets in $C^1_{d}(X/B)$.
\end{theorem}

\begin{proof}
Let us first fix the notation, $C^1_{d,d}(X_b):=C^1_d(X_b)\times C^1_d(X_b)$. The arguments present here are relative in nature that is all the fiberwise phenomena can be described in terms of relative cycles on relative Chow schemes. We refer \cite{Ko, SV} to the readers for understanding the relative cycles. 
Assume that the relation $D_b=D_b^+-D_b^-$ is rationally equivalent to zero. This means that there exists a map $f_b:\PR^1\to C^1_{d,d}(X_b)$ such that
$$f_b(0)=D_b^{+}+\gamma\text{ and }f_b(\infty)=D_b^{-}+\gamma,$$
where $\gamma$ is a positive divisor on $X_b$.
In other words, we have the following map: $$\ev:Hom^v(\PR^1_B,C^1_{d}(X/B))\to C^1_d(X/B)\times C^1_d(X/B),$$ given by $f\mapsto (f(0_B),f(\infty_B))$, where $0_B$ and $\infty_B$ are the sections of $B\to \PR^1_B$ defined by $b\mapsto (b,0)$ and $b\mapsto (b,\infty)$ respectively, and the image of $f$ intersected with $X_b$ is the image of $f_b$ contained in $C^1_{d,d}(X_b)$. Here, $Hom^v$ denotes the Hom-variety parametrizing the morphisms of certain schemes over $B$. 

Let us denote $C^1_d(X/B)$ by $C^1_d(X)$ for simplicity.
We now consider the subscheme $U_{v,d}(X)$ of $B\times Hom^v(\PR^1_B,C^1_{d}(X))$ consisting of the pairs $(b,f)$ such that the image of $f$ intersected with $X_b$ is $f_b$ (such a universal family exists, for example see \cite[Theorem 1.4]{Ko}). This gives a morphism from $U_{v,d}(X)$ to $B\times C^1_{d,d}(X)$ defined by $$(b,f)\mapsto (b,f_b(0),f_b(\infty)).$$
Again, we consider the closed subscheme $\bcV_{d,d}$ of $B\times C^1_{d,d}(X)$ given by $(b,z_1,z_2)$, where $(z_1,z_2)\in C^1_{d,d}(X_b)$. Suppose that the map from $\bcV_{d,u,d,u}$ to $\bcV_{d+u,u,d+u,u}$ is given by
$$(A,C,B,D)\mapsto (A+C,C,B+D,D).$$
Then one writes the fiber product $\bcV$ of $U_{v,d}(X)$ and $\bcV_{d,u,d,u}$ over $\bcV_{d+u,u,d+u,u}$. If we consider the projection from $\bcV$ to $B\times C^1_{d,d}(X)$, then we observe that $A$ and $B$ are supported as well as rationally equivalent on $X_b$. Conversely, if $A$ and $B$ are supported as well as rationally equivalent on $X_b$, then one gets the map $$f_b:\PR^1_b\to C^1_{d+u,u,d+u,u}(X_b)$$ of some degree $v$ satisfying
$$f_b(0)=(A+C,C)\text{ and } f_b(\infty)=(B+D,D),$$
where $C$ and $D$ are supported on $X_b$. This implies that the image of the projection from $\bcV$ to $B\times C^1_{d,d}(X)$ is a quasi-projective subscheme $W_{d}^{u,v}$ consisting of the tuples $(b,A,B)$ such that $A$ and $B$ are supported on $X_b$, and that there exists a map $$f_b:\PR^1_b\to C^1_{d+u,u}(X_b)$$ such that $f_b(0)=(A+C,C)$ and $f_b(\infty)=(B+D,D)$. Here $f_b$ is of degree $v$, and $C,D$ are supported on $X_b$ and they are of co-dimension one and degree $u$ cycles. This shows that $W_d=\cup_{u,v} W_d^{u,v}$. We now prove that the Zariski closure of $W_d^{u,v}$ is in $W_d$ for each $u$ and $v$. For this, we prove the following:
$$W_d^{u,v}=pr_{1,2}(\wt{s}^{-1}(W^{0,v}_{d+u}\times W^{0,v}_u)),$$
where
$$\wt{s}: B\times C^1_{d,d,u,u}(X)\to B\times C^1_{d+u,d+u,u,u}(X)$$
is defined by
$$\wt{s}(b,A,B,C,D)=(b,A+C,B+D,C,D).$$

We consider $(b,A,B,C,D)\in B\times C^1_{d,d,u,u}(X)$ such that $\wt{s}(b,A,B,C,D)\in W^{0,v}_{d+u}\times W^{0,v}_u$. This implies that there exists an element $(b,g)\in B\times\Hom^v(\PR^1_B,C^p_{d+u}(X))$ and an element $(b,h)\in \Hom^v(\PR^1_B,C^p_{u}(X))$ satisfying $$g_b(0)=A+C,~g_b(\infty)=B+D \text{ and } h_b(0)=C,h_b(\infty)=D$$ as well as the images of $g_b$ and $h_b$ are contained in $C^1_{d+u}(X_b)$ and  $C^1_u(X_b)$ respectively.

Also if $f=g\times h$ then $f\in \Hom^v(\PR^1_B,C^p_{d+u,u}(X))$ such that the intersection of the  image of $f$ with $X_b$ gives a closed point on $C^1_{d+u,u}(X_b)$ as well as it satisfies the following:
$$f_b(0)=(A+C,C)\text{ and }(f_b(\infty))=(B+D,D).$$
This shows that $(b,A,B)\in W^d_{u,v}$.

On the other hand, we assume that $(b,A,B)\in W^d_{u,v}$. Then there exists $f\in \Hom^v(\PR^1_B,C^1_{d+u,u}(X))$ such that
$$f_b(0)=(A+C,C)\text{ and }f_b(\infty)=(B+D,D),$$
and image of $f$ intersection $X_b$ is the image of  $f_b$.

We now compose $f$ with the projections to respectively $C^1_{d+u}(X)$ and $C^1_{u}(X)$ to get a map $g\in \Hom^v(\PR^1_B,C^1_{d+u}(X))$ and a map $h\in\Hom^v(\PR^1_B,C^1_{u}(X))$ satisfying
$$g_b(0)=A+C,~g_b(\infty)=B+D$$
and
$$h_b(0)=C,~h_b(\infty)=D.$$
Also, the images of $g_b$ and $h_b$ are contained in the respective Chow schemes of the fibers $X_b$. Therefore, we have
$$W_d=pr_{1,2}(\wt{s}^{-1}(W_{d+u}\times W_u)).$$
We are now in a position to prove that the closure of $W_d^{0,v} $ is contained in $W_d$. Let $(b,A,B)$ be a closed point in the closure of ${W_d^{0,v}}$. Let $W$ be an irreducible component of ${W_d^{0,v}}$ whose closure contains $(b,A,B)$. Assume that $U$ is an affine neighborhood of $(b,A,B)$ such that $U\cap W$ is non-empty. Then there is an irreducible curve $C$ in $U$ passing through $(b,A,B)$. Suppose that $\bar{C}$ is the Zariski closure of $C$ in $\overline{W}$. The map
$$e:U_{v,d}(X)\subset B\times \Hom^v(\PR^1_B,C^1_{d}(X))\to C^1_{d,d}(X)$$
given by
$$(b,f)\mapsto (b,f_b(0),f_b(\infty))$$
is regular and $W_d^{0,v}$ is its image. We now choose a curve $T$ in $U_{v,d}(X)$ such that the closure of $e(T)$ is $\bar C$.  Let $\wt{T}$ be denote the normalization of the Zariski closure of $T$, and $\wt{T_0}$ the pre-image of $T$ in this normalization. Then the regular morphism $\wt{T_0}\to T\to \bar C$ extends to a regular morphism from $\wt{T}$ to $\bar C$. If $(b,f)$ is a pre-image of $(b,A,B)$, then $f_b(0)=A,~ f_b(\infty)=B$ and the image of $f_b$ is contained in $C^p_{d}(X_b)$ by the definition of $U_{v,d}(X)$. Therefore, $A$ and $B$ are  rationally equivalent. This completes the proof.
\end{proof}

As a consequence, for a fixed positive integer $n$ which need not be the same as $d$, one gets the following:

\begin{corollary}\label{cor2.1} The collection
$$\bcZ_{n,d}:=\{(b,D_b)|n[D_b]=0 \in \CH^1(X_b)\}$$
is a countable union of Zariski closed subsets in the scheme $C^1_d(X/B)$. 
\end{corollary}

\section{Mumford-Roitman arguments and monodromy representation}
\label{section monodromy}

Throughout this section we fix a positive integer $n$.
The main  idea of this section, is due to Voisin \cite[Chapter 3]{Vo}. We consider the case when $X$ is a smooth, projective algebraic surface over $\bar \QQ$ fibered over a projective algebraic curve over $\bar \QQ$. Then the fibers of the map $X\to B$ are projective algebraic curves. Moreover, assume that the fiber $X_b$ is smooth for a general $b\in B$. We now consider the complexification of the above family. That is, we consider the map $X_{\CC}\to B_{\CC}$, where $X_{\CC}=X\times _{\bar \QQ}\CC$. Then there is a map $$\varphi:\CH^1(X_b)\to\CH^1(X_{b\CC})$$ given by pull back. The set of all images of the  $n$-torsions on $\CH^1(X_b)$ under $\varphi$ is a subgroup of the group of torsions in $\CH^1(X_{b\CC})$. Assume  $\bcZ_{n,d}$ as defined in Corollary \ref{cor2.1}. Consider the natural morphism from $C^1_d(X/B)$ to $\Pic^0(X/B)$ and continue to call the image of $\bcZ_{n,d}$ as $\bcZ_{n,d}$ under this morphism. To construct $\Pic^0$ we need to remove the singular fibers of $X\to B$. After removing the points from $B$ corresponding to the singular fibers we obtain a Zariski open subset of $B$, which we call as $B$. Let $\bcZ_n=\cup_i \bcZ_{n,i}\subset \Pic^0(X/B)$. Then $$\bcZ_n=\{(D,b)\mid n[D]=0\}.$$
We consider the map $\bcZ_n\to B$. Then for any $b\in B$ there exists a $D$ supported on $X_b$ such that $n[D]=0$ on $\CH^1(X_b)$. If $B_{\CC}$ is connected then there exists a $\bcZ_{n,i}$ such that the map from $\bcZ_{n,i\CC}$ to $B_{\CC}$ is onto. Therefore, if $U$ is the Zariski open set in $B$ consisting of $b$ such that $X_b$ is smooth then one gets a surjective map $\mathfrak{f}: \bcZ_{n,i\CC,U}\to U_{\CC}$. By removing a finite collection of points from $U_{\CC}$, one can assume that the map $\mathfrak{f}$ is smooth and proper in the sense of underlying smooth manifolds. Hence, $\mathfrak{f}$ is a fibration in the sense of differential topology. It is easy to see that $\mathfrak{f}$ has finite fibers since these fibers contain $n$-torsion points on $\CH^1(X_{b,\CC})$. In fact, the $n$-torsion points are inside $\Pic^0(X_{b\CC})$, which is an abelian variety isomorphic to $J(X_{b,\CC})$, the Jacobian variety of $X_{b\CC}$.

We now consider the first cohomology $H^1(J(X_{b\CC}),\QQ)$ which is also isomorphic to $H^1(X_{b,\CC},\QQ)$. We also consider the cohomology classes in $H^1(J(X_{b\CC}),\QQ)$ of the elements in the fibers of $\mathfrak{f}$, and let  $V_b$ be its $\QQ$-span. Since  this map $\mathfrak{f}$ is smooth and proper, the cohomology classes in $V_b$ gives rise to a $\pi_1(U_{\CC},b)$ module which is a $\pi_1(U_{\CC},b)$ submodule of $H^1(X_{b,\CC},\QQ)$. Therefore, the cohomology classes of the torsion points on the fibers of $\bcZ_{n,i\CC,U}\to U_{\CC}$ give rise to a locally constant sheaf $\bcL$ over $U_{\CC}$ by the equivalence between locally constant sheaves and $\pi_1(U_{\CC},b)$-module representations. Hence, the dimension of the vector space $V_b$ over $\QQ$ remains constant as $U$ is connected. Thus, the cardinality of the finite group of $n$-torsions coming from $\bcZ_{n,i\CC,U}\to U_{\CC}$ in $\Pic(X_{b\CC})$ remains constant. Therefore the cardinality of $n$-torsions in $\Pic^0(X_{b\CC})$, which are defined over $\bar K$ is also constant. This is due to $\Pic^0(X_{b\CC})\cong {\Pic^0(X_b)}_{\CC}$. This gives the following result.

\begin{theorem} \label{theorem2}
The cardinality of the subgroup of $n$-torsions in $X_b$ arising from the fibration $\bcZ_{n,i\CC,U}\to U_{\CC}$ for each $b\in U$ remains constant and they vary in a family.
\end{theorem}

We now consider a smooth projective curve $C$ over an algebraically closed field $K\subset \CC$ in the  projective plane $\PR^2$ over $K$. Let $U$ be an affine piece of $C$. That is, $U$ is $C$ minus finitely many points, viz. $P_1,\cdots,P_m$. Consider the following localization exact sequence of Picard groups
$$\oplus_i \ZZ[P_i]\to \CH^1(C)\to \CH^1(U)\to 0.$$

Fix some positive integer $n$. Then the set of all $n$-torsion points in $\CH^1(U)$ gives rise to elements of $\CH^1(C)$ of the form $nz$ such that $nz=\sum_i n_iP_i$, where $P_1,\cdots, P_m$ are the finite number of points that are deleted. As before, we consider a fibration of smooth projective schemes $X\to B$ over $\bar \QQ$, where $X$ is a surface embedded in $\PR^3$ such that each fiber $X_b$ is contained in a projective plane $\PR^2$ over $\bar \QQ$ and $B$ is an algebraic curve. Suppose that the degree of the algebraic curve $X_b$ remains constant over a Zariski open set $U$ in $B$. For an affine piece $U_b$ of the algebraic curve $X_b$, we consider the following:
$$\bcP:=\{(x,b)|x\in X_b\setminus U_b\}\to U.$$
Here $x\in X_b\setminus U_b$ are the points $P_{1b}, \cdots, P_{mb}$ such that we have 
$$nz=\sum_i P_{ib}\;.$$
After shrinking $U$, assume that $X_U\to U$ is finite of a constant degree. Then the number of points in $X_b\setminus U_b$ is finite (which vary constantly with $b$ in fibers). Therefore the number of zero cycles like above must give a finitely many torsion points of $\CH^1(U_b)$. So the torsion points on $\CH^1(U_b)$ corresponds to the complement of one point in $C^1_d(X_b)$ as $b$ varies, which describes the Zariski closed subset $\bcP$ of $X_U\times U$.

By the above assumption, this a finite-to-one map from $\bcP$ to $U$ and the degree of this map is constant. For a given $b\in U$, let us assume that the fiber $\bcP_b$ contains the points $P_{1b},\cdots,P_{mb}$. Fix a positive integer $n$ and  define the set:
$$\bcZ_d=\{(b,z)\in U\times C^1_d(X_U/U)|\Supp(z)\subset X_b, n[z]=\sum_i n_i[P_{ib}]\}\to U.$$
Then as a conseqeunce of Theorem \ref{theorem1} one gets the following result.
\begin{corollary}
$\bcZ_d$ is a countable union of Zariski closed subsets in the ambient relative Chow variety $C^1_d(X_U/U)$, where $X_U\to U$ is the pullback of the family $X\to B$ to $U$.
\end{corollary}
We consider the map $\vartheta_d:C^1_d(X_U/U)\to \CH^1(X_U)$ and the union
$$\vartheta_d(\bcZ_d)=\bcZ=\cup_i \bcZ_i.$$
Consider the image of $\bcZ_d$ in $\Pic^0(X_U/U)$, under the natural morphism from $C^1_d(X_U/U)$. Continue to call it $\bcZ_d$. Consider the union $\cup_i \bcZ_i$ in $\Pic^0(X_U/U)$. Since $U$ is connected, so that there exists at least one $i$ such that $\bcZ_{i\CC}\to U_{\CC}$ is surjective. Therefore as in Theorem \ref{theorem2}, it gives rise to a local system. More precisely, for each $b\in U$ if one considers the $\pi_1(U_{\CC},b)$-module $V_b$ consisting of cohomology classes of the elements in the fibers of the map $\bcZ_{i\CC}\to U_{\CC}$ in $H^1(X_{b\CC},\QQ)$, then the  dimension of this vector space $V_b$ remains constant over $U$.  Therefore, we have the following:
\begin{corollary} \label{theorem3}
The cardinality of the set of $z$ in $\bcZ_{ib}$ for $b\in U$ such that
$$n[z]=\sum_i n_i [P_{ib}]$$
for points $P_{ib}\in X_b$ is constant as $b$ varies over $U$.
\end{corollary}
These points on $\CH^1(X_b)$ are corresponding to the torsion elements in $\CH^1(U_b)$, where $U_b$ is the open complement of $X_b$ obtained from $X_b$ by  deleting the points $P_{1b},\cdots, P_{mb}$.

\section{Torsion in the Picard group of algebraic surfaces}

Throughout the section we fix a positive integer $n$. We begin with a certain algebraic surface having an $n$-torsion element in the Weil divisor class group. Consider the algebraic surface defined by
$$y^2=t^2q^2-z^n$$
over $\QQ$. Its co-ordinate ring is given by
$$\QQ[y,t,z]/(y^2-t^2q^2+z^n).$$
We now consider the maximal ideal $(t-m,z-\ell)$, for some  algebraic  numbers $m,\ell$,  in the polynomial ring $\QQ[t,z]$. We also consider the map $$\QQ[t,z]\to \QQ[y,t,z]/(y^2-t^2q^2+z^n)$$
defined by
$$t\mapsto t, z\mapsto z$$
and the map $\QQ[t,z]\to \QQ$, which is given by
$$f(t,z)\mapsto f(m,\ell).$$
Then the tensor product
$$ \QQ[y,t,z]/(y^2-t^2q^2+z^n)\otimes _{\QQ[t,z]}\QQ$$
is given by $ \QQ[y]/(y^2-m^2q^2+\ell^n)$. Further, if the polynomial $p(y):=y^2-m^2q^2+\ell^n$ is irreducible over $\QQ$, then the above co-ordinate ring is isomorphic to $L$, where $L$ is the imaginary quadratic extension of $\QQ$ given by adjoining a root of $p(y)$. Therefore, if we consider the family $$\ZZ[y,t,z]/(y^2-t^2q^2+z^n)\to \ZZ[t,z],$$
then the normalizations of the fibers are the ring of integers of
$\QQ(\sqrt {m^2q^2-\ell^n})\;.$
These normalizations form a family in the following way.

Let us consider the normalization $\wt{S}$ of $S$, that is, the spectrum of the integral closure of the co-ordinate ring of $S$. If we spread out $\wt{S}$ over $\Spec(\ZZ)$, then it's co-ordinate ring is isomorphic to
$\ZZ[z,t](\alpha)$,
where $\alpha$ is a root of
$y^2-t^2q^2+z^n$.
Therefore when we consider the fibers of this spread $\wt{S}_{\ZZ}$ over $t=m,z=\ell$, then we get the spectrum of the ring
$$\ZZ[y](\sqrt{m^2q^2-\ell^n})\;.$$
This is the ring  of integers of $\QQ(\sqrt{m^2q^2-\ell^n})$. Denote this family by $\wt{S}_{\ZZ}$, and  consider the Zariski closure of $\wt{S}$ (the generic fiber of the above mentioned family) in $\PR^3_{\QQ}$ and the Zariski closure of the family $\wt{S_{\ZZ}}\to \AF^2_{\ZZ}$ in $\PR^3_{\ZZ}$. We denote it by $\bar {\wt{S}_{\ZZ}}$. We also consider the Chow variety
$C^1_d(\bar{\wt{S}_{\ZZ}}/\PR^2_{\ZZ})$
and the subset of $C^1_d(\bar{\wt{S}_{\ZZ}}/\PR^2_{\ZZ})\times _{\ZZ}\PR^2_{\ZZ}$
$$\bcZ_d:=\{(z,b)|\Supp(z)\subset \bar{\wt{S}_{\ZZ,b}}, [z]=\sum_i n_i[P_{ib}]\in \CH^1(\bar{\wt{S}_{\ZZ,b}})\},$$
where $P_{1b},\cdots,P_{mb}$ are the points in the complement of $\wt{S_{\ZZ,b}}$ inside the Zariski closure $\bar {\wt{S}_{\ZZ,b}}$. Then by Theorem \ref{theorem1}, we get the following result\footnote{Here, the notion of Chow group of co-dimension one subvarieties of an arithmetic variety is in the sense of \cite{GS}(section 3.3, subsection 3.3.1)}.
\begin{proposition}
The set $\bcZ_d$ is a countable union of Zariski closed subsets in the Chow variety.
\end{proposition}
This will follow if we can prove the analog of theorem \ref{theorem1} for arithmetic varieties, which is the following:
\begin{theorem} The set
$$\bcZ_d:=\{(z,b)|\Supp(z)\subset \bar{\wt{S}_{\ZZ,b}}, [z]=0\in \CH^1(\bar{\wt{S}_{\ZZ,b}})$$ is a countable union of Zariski closed subschemes in Chow variety.
\end{theorem}
\begin{proof}
The proof goes along the line of \ref{theorem1}. The only points to be noted here are:
\begin{itemize}
\item[(I)] The notion of Hilbert scheme and the Hom-scheme make sense for an arithmetic variety. This is as explained in \cite[Chapter: Hilbert schemes and Quot schemes, \S 5]{FGA}.

\item[(II)] The family of Weil divisors of a smooth fibration over $\Spec(\ZZ)$ are parametrized by a Chow variety which is actually given by the Picard scheme parametrizing relative Cartier divisors of the same family \cite[Corollary 11.8]{Ry}. In our case the family is $\wt{S}_{\ZZ}$ which is of finite presentation over $\ZZ$ and it is a standard smooth algebra\footnote{In the sense,  \cite[\href{https://stacks.math.columbia.edu/tag/01UM}{Sections 01UM} \& \href{https://stacks.math.columbia.edu/tag/00SK}{00SK}]{stacks-project}, see Definitions 10.136.6 \& 29.32.1.} over $\ZZ$. This enables us to formulate the definition of rational equivalence for arithmetic varieties as in \cite[\S 3.3]{GS} in the way: two Weil divisors $D_1,D_2$ are rationally equivalent if there exists a morphism $f: \PR^1_{\ZZ}\to C^1_{d,d}(\bar{\wt{S}_{\ZZ}}/\PR^2_{\ZZ})$ such that
$f(0)=D_1+B$ and $ f(\infty)=D_2+B$,
where $B$ is a positive Weil divisor and $0,\infty$ are two fixed points on $\PR^1_{\ZZ}$.
\end{itemize}
The rest of the proof goes mutatis mutandis as in \ref{theorem1}. Although it is similar we present it for the convenience of the reader.

Assume that the relation $D_b=D_b^+-D_b^-$ is rationally equivalent to zero. This means that there exists a map $f:\PR^1\to C^1_{d,d}(\bar{\wt{S}_{\ZZ,b}})$ such that
$$f(0)=D_b^{+}+\gamma\text{ and }f(\infty)=D_b^{-}+\gamma,$$
where $\gamma$ is a positive divisor on $\bar{\wt{S}_{\ZZ,b}}$.
In other words, we have the following map: $$\ev:Hom^v(\PR^1_{\ZZ},C^1_{d}(\bar{\wt{S}_{\ZZ}}/\PR^2_{\ZZ}))\to C^1_{d}(\bar{\wt{S}_{\ZZ}}/\PR^2_{\ZZ})\times C^1_{d}(\bar{\wt{S}_{\ZZ}}/\PR^2_{\ZZ}) ,$$ given by $f\mapsto (f(0),f(\infty))$ and image of $f$ is contained in $C^1_{d,d}(\bar{\wt{S}_{\ZZ,b}})$.

Let us denote $C^1_{d}(\bar{\wt{S}_{\ZZ}}/\PR^2_{\ZZ})$ by $C^1_d(\bar{\wt{S}}_{\ZZ})$ for simplicity.
We now consider the subscheme $U_{v,d}(\bar{\wt{S}_{\ZZ}})$ of $\PR^2_{\ZZ}\times \Hom^v(\PR^1_k,C^1_{d}(\bar{\wt{S}_{\ZZ}}))$ consisting of the pairs $(b,f)$ such that image of $f$ is contained in $C^1_{d}(\bar{\wt{S}_{\ZZ,b}})$ (such a universal family exists, for example see \cite[Theorem 1.4]{Ko} or \cite{FGA}, chapter on Hilbert schemes and Quot schemes). This gives a morphism from $U_{v,d}(\bar{\wt{S}_{\ZZ}})$ to
$$\PR^2_{\ZZ}\times C^1_{d,d}(\bar{\wt{S}_{\ZZ,b}})$$
 defined by $$(b,f)\mapsto (b,f(0),f(\infty)).$$
Again, we consider the closed subscheme $\bcV_{d,d}$ of $\PR^2_{\ZZ}\times C^1_{d,d}(\bar{\wt{S}_{\ZZ}})$ given by $(b,z_1,z_2)$, where $(z_1,z_2)\in C^1_{d,d}(\bar{\wt{S}_{\ZZ,b}})$. Suppose that the map from $\bcV_{d,u,d,u}$ to $\bcV_{d+u,u,d+u,u}$ is given by
$$(A,C,B,D)\mapsto (A+C,C,B+D,D).$$
Then one writes the fiber product $\bcV$ of $U_{v,d}(\bar{\wt{S}_{\ZZ}})$ and $\bcV_{d,u,d,u}$ over $\bcV_{d+u,u,d+u,u}$. If we consider the projection from $\bcV$ to $\PR^2_{\ZZ}\times C^1_{d,d}(\bar{\wt{S}_{\ZZ}})$, then we observe that $A$ and $B$ are supported as well as rationally equivalent on $\bar{\wt{S}_{\ZZ,b}}$. Conversely, if $A$ and $B$ are supported as well as rationally equivalent on $\bar{\wt{S}_{\ZZ,b}}$, then one gets the map $$f:\PR^1_{\ZZ}\to C^1_{d+u,u,d+u,u}(\bar{\wt{S}_{\ZZ,b}})$$ of some degree $v$ satisfying
$$f(0)=(A+C,C)\text{ and } f(\infty)=(B+D,D),$$
where $C$ and $D$ are supported on $\bar{\wt{S}_{\ZZ,b}}$. This implies that the image of the projection from $\bcV$ to $\PR^2_{\ZZ}\times C^1_{d,d}(\bar{\wt{S}_{\ZZ}})$ is a quasi-projective subscheme $W_{d}^{u,v}$ consisting of the tuples $(b,A,B)$ such that $A$ and $B$ are supported on $\bar{\wt{S}_{\ZZ,b}}$, and that there exists a map $$f:\PR^1_{\ZZ}\to C^1_{d+u,u}(\bar{\wt{S}_{\ZZ,b}})$$ such that $f(0)=(A+C,C)$ and $f(\infty)=(B+D,D)$. Here $f$ is of degree $v$, and $C,D$ are supported on $\bar{\wt{S}_{\ZZ,b}}$ and they are of co-dimension one and of degree $u$ cycles. This shows that $W_d=\cup_{u,v} W_d^{u,v}$. We now prove that the Zariski closure of $W_d^{u,v}$ is in $W_d$ for each $u$ and $v$. For this, we prove the following:
$$W_d^{u,v}=pr_{1,2}(\wt{s}^{-1}(W^{0,v}_{d+u}\times W^{0,v}_u)),$$
where
$$\wt{s}: \PR^2_{\ZZ}\times C^1_{d,d,u,u}(\bar{\wt{S}_{\ZZ}})\to \PR^2_{\ZZ}\times C^1_{d+u,d+u,u,u}(\bar{\wt{S}_{\ZZ}})$$
defined by
$$\wt{s}(b,A,B,C,D)=(b,A+C,B+D,C,D).$$

We assume $(b,A,B,C,D)\in \PR^2_{\ZZ}\times C^1_{d,d,u,u}(\bar{\wt{S}_{\ZZ}})$ in such a way that $\wt{s}(b,A,B,C,D)\in W^{0,v}_{d+u}\times W^{0,v}_u$. This implies that there exists an element $(b,g)\in \PR^2_{\ZZ}\times\Hom^v(\PR^1_k,C^p_{d+u}(\bar{\wt{S}_{\ZZ}}))$ and an element $(b,h)\in \Hom^v(\PR^1_{\ZZ},C^p_{u}(\bar{\wt{S}_{\ZZ}}))$ satisfying $$g(0)=A+C,~g(\infty)=B+D \text{ and } h(0)=C,h(\infty)=D$$ as well as the image of $g$ and $h$ are contained in $C^1_{d+u}(\bar{\wt{S}_{\ZZ,b}})$ and  $C^1_u(\bar{\wt{S}_{\ZZ,b}})$ respectively.

Also if $f=g\times h$ then $f\in \Hom^v(\PR^1_{\ZZ},C^p_{d+u,u}(\bar{\wt{S}_{\ZZ}}))$ such that the image of $f$ is contained in $C^1_{d+u,u}(\bar{\wt{S}_{\ZZ,b}})$ as well as it satisfies the following:
$$f(0)=(A+C,C)\text{ and }(f(\infty))=(B+D,D).$$
This shows that $(b,A,B)\in W^d_{u,v}$.

On the other hand, we assume that $(b,A,B)\in W^d_{u,v}$. Then there exists $f\in \Hom^v(\PR^1_k,C^1_{d+u,u}(\bar{\wt{S}_{\ZZ,b}}))$ such that
$$f(0)=(A+C,C)\text{ and }f(\infty)=(B+D,D),$$
and image of $f$ is contained in the Chow scheme of  $\bar{\wt{S}_{\ZZ,b}}$.

We now compose $f$ with the projections to $C^1_{d+u}(\bar{\wt{S}_{\ZZ,b}})$ and to $C^1_{u}(\bar{\wt{S}_{\ZZ,b}})$ to get a map $g\in \Hom^v(\PR^1_{\ZZ},C^1_{d+u}(\bar{\wt{S}_{\ZZ}}))$ and a map $h\in\Hom^v(\PR^1_{\ZZ},C^1_{u}(\bar{\wt{S}_{\ZZ}}))$ satisfying
$$g(0)=A+C,~g(\infty)=B+D$$
and
$$h(0)=C,~h(\infty)=D.$$
Also, the images of $g$ and $h$ are contained in the respective Chow schemes of   the fibers $\bar{\wt{S}_{\ZZ,b}}$. Therefore, we have
$$W_d=pr_{1,2}(\wt{s}^{-1}(W_{d+u}\times W_u)).$$

We are now in a position to prove that the closure of $W_d^{0,v} $ is contained in $W_d$. Let $(b,A,B)$ be a closed point in the closure of ${W_d^{0,v}}$. Let $W$ be an irreducible component of ${W_d^{0,v}}$ whose closure contains $(b,A,B)$. Assume that $U$ is an affine neighborhood of $(b,A,B)$ such that $U\cap W$ is non-empty. Then there is an irreducible curve $C$ in $U$ passing through $(b,A,B)$. Suppose that $\bar{C}$ is the Zariski closure of $C$ in $\overline{W}$. The map
$$e:U_{v,d}(\bar{\wt{S}_{\ZZ}})\subset \PR^2_{\ZZ}\times \Hom^v(\PR^1_{\ZZ},C^1_{d}(\bar{\wt{S}_{\ZZ}}))\to C^1_{d,d}(\bar{\wt{S}_{\ZZ}})$$
given by
$$(b,f)\mapsto (b,f(0),f(\infty))$$
is regular and $W_d^{0,v}$ is its image. We now choose a curve $T$ in $U_{v,d}(X)$ such that the closure of $e(T)$ is $\bar C$.  Let $\wt{T}$ be denote the normalization of the Zariski closure of $T$, and $\wt{T_0}$ the pre-image of $T$ in this normalization. Then the regular morphism $\wt{T_0}\to T\to \bar C$ extends to a regular morphism $\wt{T}_{\QQ}\to \bar C_{\QQ}$ when we consider all the above varieties scalar extended by the field of all algebraic numbers, $\bar\QQ$. If $(b_{\QQ},f_{\QQ})$ is a pre-image of $(b_{\QQ},A_{\QQ},B_{\QQ})$, then
$$f_{\QQ}(0)=A_{\QQ}, \quad f_{\QQ}(\infty)=B_{\QQ}$$ and the image of $f_{\QQ}$ is contained in $C^p_{d}(\bar{\wt{S}})_{\QQ}$. 
Spreading out $f_{\QQ}$ over $\bar\ZZ$, we have an $f$ such that $$f(0)=A, \quad f(\infty)=B\;.$$
This is because there is a one to one correspondence between $\bar\ZZ$ (the ring of all algebraic integers) points of arithmetic varieties and $\bar\QQ$ points of the corresponding variety over $\bar\QQ$. Therefore, $A$ and $B$ are  rationally equivalent. This completes the proof.
\end{proof}

Applying the same argument as in Corollary \ref{theorem3}, we see that there exists an irreducible Zariski closed subset $\bcZ_i$ inside the  Picard scheme $\Pic(\bar{\wt{S}_{\QQ}})_V$ (this again uses the representability of family of Weil divisors in terms of relative Cartier divisors \cite{Ry}[corollary 11.8]), where $V$ is Zariski open in $\AF^2_{\QQ}$ such that the complexification of $\bcZ_{i,\CC}$ maps dominantly onto $V_{\CC}$ as well as the number of points in the fiber of this map is constant and they are the $n$-torsions. Note that the $\CH^1(\bar {\wt{S}_{\QQ}})$ is the co-limit of the groups $\CH^1(\bar{\wt{S}_{U}})$, where $U$ varies in $\AF^2_{\ZZ}$. Consider the spread $\widetilde{\bcZ_i}$ of $\bcZ_{i}$ in $\Pic(\bar {\wt{S}_U})$ and the spread of $V$ in $\AF^2_{\ZZ}$ denoted as $\bcV$. Since the map
$\bcZ_{i,\CC}\to V_{\CC}$
is dominant, so that
$\widetilde{\bcZ_{i,\CC}}\to \bcV_{\CC}$
is dominant too. Therefore applying the arguments as in \S\ref{section monodromy} one gets the following:
\begin{theorem}\label{thm3.5}
There exists a Zariski open $U$ in $\AF^2_{\ZZ}$, such that the cardinality of a certain subgroup of $n$-torsions of $\CH^1(\wt{S_{\ZZ,b}})$ remains constant as $b$ varies over $U$. However, the group $\CH^1(\wt{S_{\ZZ,b}})$ is nothing but the class group of the quadratic field $\QQ(\sqrt{m^2q^2-\ell^n})$ for some fixed integers $m$ and $\ell$.
\end{theorem}
This concludes that given an element of order $n$ in $\CH^1(\wt{S_{\ZZ,b}})$, for $b$ in $U$, one can find an element of the same order in $\CH^1(\wt{S_{\ZZ,b'}})$ for some $b'\in U$ which is different from $b$.

\section{Class group of $\mathbb{Q}(\sqrt{k^2-p^n})$}
In this section, we construct a subgroup in the class group of $\QQ(\sqrt{m^2q^2-\ell^2})$ which is isomorphic to a finite group of order $n$. More precisely, we prove:
\begin{thm}\label{T1}
Let $k\geq 1$ be an integer, $n\geq 3$ an odd integer and $p$ an odd prime such that $\gcd(k, p)=1$ and $k^2<p^n$. Let $-d~(<-3)$ be the square-free part of $k^2-p^n$. Then the class group of $K_{k,p,n}=\mathbb{Q}(\sqrt{-d})$ has an element of order $n$ if the following conditions hold:
\begin{itemize}
\item[(i)] $k\not\equiv \pm 1\pmod d$.
\item[(ii)] $t^q \not \equiv \pm k \pmod {d}$ for any proper divisor $t$ of $k$ and any prime divisor $q$ of $n$.
\item[(iii)] $p^{n/3}\neq (k'^3+2k)/3$ when $d \equiv 3 \pmod 4$, $3\mid n$ and $k'\mid k$ with $k'$ an odd positive integer.
\end{itemize}
\end{thm}
\begin{remark}
By putting $k=q$ an odd prime in Theorem \ref{T1}, we obtain \cite[Theorem 1.1]{CHKP}.
\end{remark}
\begin{remark}
Theorem \ref{T1} also gives \cite[Theorem 4.1]{CHKP} when $k=1$.
\end{remark}
We begin with the following crucial proposition on the solutions in positive integers $x$ and $n$ of the following equation,
\begin{equation}\label{de1}
dx^2+k^2=p^n,
\end{equation}
where $d\geq 2$ and $ k\geq 1$ are fixed integers such that $\gcd(k, d)=1$ and $p$ is odd prime number.
\begin{prop}\label{PB} Let $k, n, p$ and $d$ be as in Theorem \ref{T1}. Then
the equation \eqref{de1} has at most one solution in positive integer $x$ and $n$, except $(x, n)=(11, 5)$ which occurs only when $d=2$ and $k=1$.
\end{prop}
In order to prove this proposition, we need to recall the following result of Bugeaud and Shorey \cite{BS01}. Before stating this result, one needs to introduce some definitions and notations.

Let $F_\ell$ be denote the $\ell$-th term in the Fibonacci sequence defined by $F_{\ell+2}=F_\ell+F_{\ell+1}$ for $\ell\geq 0$ with the initials $F_0=0$ and $F_1=1$. Analogously, $L_\ell$ denotes the $\ell$-th term in the Lucas
sequence defined by $L_{\ell+2}=L_\ell+L_{\ell+1}$ for $\ell\geq 0$ with the initials $L_0=2$ and $L_1=1$.
For $\lambda\in \{1, \sqrt{2}, 2\}$, we define the subsets $\mathcal{F}, \ \mathcal{G}_\lambda, ~ \mathcal{H}_\lambda\subset \mathbb{N}\times\mathbb{N}\times\mathbb{N}$ by
\begin{align*}
\mathcal{F}&:=\{(F_{\ell-2\varepsilon},~L_{\ell+\varepsilon},~F_\ell)\mid
\ell\geq 2,~ \varepsilon\in\{\pm 1\}\},\\
\mathcal{G}_\lambda&:=\{(1,4p^r-1,p)\mid \text{$p$ is an odd prime and } r\geq 1 \text{ is an integer}\},\\
\mathcal{H}_\lambda&:=\left\{(D_1,D_2,p)\,\left|\,
\begin{aligned}
&\text{$D_1$, $D_2$ and $p$ are mutually coprime positive}\\
&\text{integers with $p$ an odd prime and there exist}\\
&\text{positive integers $r$, $s$ such that $D_1s^2+D_2=\lambda^2p^r$}\\
&\text{and $3D_1s^2-D_2=\pm\lambda^2$}
\end{aligned}\right.\right\},
\end{align*}
except when $\lambda =2$, in which case the condition ``odd''
on the prime $p$ should be removed from the definitions of $\mathcal{G}_\lambda$
and $\mathcal{H}_\lambda$.

\begin{thma}[Bugeaud and Shorey, \cite{BS01}]\label{BST}
Given $\lambda\in \{1, \sqrt{2}, 2\}$, a prime $p$ and positive co-prime integers $D_1$ and $D_2$, the number of positive integer solutions $(x, y)$ of the equation
 \begin{equation}\label{BSe}
  D_1x^2+D_2=\lambda^2p^y
 \end{equation}
is at most one except for $$
(\lambda,D_1,D_2,p)\in\mathcal{S}:=\left\{\begin{aligned}
&(2,13,3,2),(\sqrt 2,7,11,3),(1,2,1,3),(2,7,1,2),\\
&(\sqrt 2,1,1,5),(\sqrt 2,1,1,13),(2,1,3,7)
\end{aligned}\right\}
$$
and $(D_1, D_2, p)\in
\mathcal{F}\cup \mathcal{G_\lambda}\cup \mathcal{H_\lambda}$.
\end{thma}
We also recall the following result of Ljunggren \cite{LJ43} which will be needed to prove Proposition \ref{PB}.
\begin{thma}\label{WT}
For an odd integer $n$, the only solution to the equation
\begin{equation}
\frac{y^n-1}{y-1}=x^2
\end{equation}
in positive integers $x,y$ and $n$ with $y>1$ is $(11, 3, 5)$.
\end{thma}
We also need the following result of Cohn \cite{Cohn1} which talks about appearance of squares in the Lucas sequence.
\begin{thma}\label{CT}
The only perfect squares appear in the Lucas sequence are $L_1=1$ and $L_3=4$.
\end{thma}
\begin{proof}[\bf Proof of Proposition \ref{PB}]
Here, $(\lambda, D_1, D_2, p)=(1, d, k^2, p)$. Thus if   $(1, d, k^2, p) \in \mathcal{S}$, then  $d=2, k=1, p=3$, and hence \eqref{de1} gives
$2x^2+1=3^n$. Therefore by putting $y=3$ in \eqref{WT} we see that $(x, n)=(11, 5)$ is the only solution in positive integer of this equation.

We now assume that $(d, k^2, p) \in \mathcal{F}$. Then $L_{\ell+\varepsilon}=k^2$, and thus by Theorem \ref{CT} only possible values of $\ell$ are $0, 1, 2$ and $4$. Again utilizing $F_\ell=p$, one gets $\ell=4$ which corresponds to $p=3,~ k=2$ and $\varepsilon=-1$. Therefore by using $F_{\ell-2\varepsilon}=d$, we get $d=8$. This shows that $2\mid 3$ by \eqref{de1} since $k=2$ and $p=3$. Hence $(d, k^2, p) \not\in \mathcal{F}$.

Again, if $(d,k^2,p)\in \mathcal{G}_1$ then $4p^r-1=k^2$ for some integer $r\geq 1$. This shows that $k^2\equiv 3\pmod 4$ which is not true.

Finally, if $(d, k^2,p) \in \mathcal{H}_1$ then there are positive integers $r$ and $s$ such that
\begin{equation}\label{eq3.4}
3ds^2-k^2=\pm 1
\end{equation}
and
\begin{equation}\label{eq3.5}
ds^2+k^2=p^r.
\end{equation}
One can conclude by reading \eqref{eq3.4} modulo $3$ that $3ds^2-k^2=1$ is not possible, and thus  \ref{eq3.4}) becomes,
$$3ds^2-k^2=- 1.$$
This equation together with (\ref{eq3.5}) would  imply
$$4k^2=3p^r+1.$$
This further implies $$(2k-1)(2k+1)=3p^r.$$
This leads to the following possibilities (since $p$ is prime):
$$\begin{cases}
2k-1=1,\\
2k+1=3p^r,
\end{cases}$$
or
$$\begin{cases}
2k-1=3,\\
2k+1=p^r.
\end{cases}$$
The first pair of equations imply $2=3p^r-1$ which is not possible. Again, the last pair of equations implies $k=2,~ p=5$ and $r=1$. Thus \eqref{eq3.5} gives $d=1$ which is a contradiction. Hence by Theorem \ref{BST}, we complete the proof.
\end{proof}
We now recall the following result \cite[Proposition 2.1]{CHKP} which will be needed in the proof of the next result.
\begin{thma}\label{PC}
Let $d\equiv 5\pmod 8$ be an integer and $q$ a prime. For any odd integers $a$ and $b$, the following holds:
$$\left(\frac{a+b\sqrt{d}}{2}\right)^q\in \mathbb{Z}(\sqrt{d})\text{ if and only if } q=3.$$
\end{thma}
We now prove the following result which is the main ingredient in the proof of Theorem \ref{T1}.
\begin{prop}\label{P2}
Let $k,p, n, d$ be as in Theorem \ref{T1} and $m$ the positive integer such that $k^2-p^n=-m^2d$. Then $\alpha =k+m\sqrt{-d}$ is not an $q^{th}$ power of an element in the ring of integers, $\mathcal{O}_{K_{k,p, n}}$ of the field $K_{k,p, n}$ for any prime divisor $q $ of $n$.
\end{prop}

\begin{proof}
Let $q$ be a prime number such that $q\mid n$. Then $q$ is odd since $n$ is odd.

We first consider the case when $d \equiv 1, 2 \pmod 4$. Suppose $\alpha=\left(a+b\sqrt{-d}\right)^\ell$ for some integers $a$ and $b$. That is, $$k+m \sqrt{-d} =(a+b\sqrt{-d})^q.$$
We now compare the real parts to get,
\begin{equation}\label{eq1}
k=a^q+\sum_{\ell=1}^{\frac{q-1}{2}} \binom{q}{2\ell} a^{q-2\ell}b^{2\ell}(-d)^\ell.
\end{equation}
This shows that $a\mid k$ and thus $a=\pm 1, \pm t$ or $ \pm k$ for some proper divisor $t$ of $k$.\\
If $a=\pm 1$, then \eqref{eq1} implies
$$k=\pm 1\pm \sum_{\ell=1}^{\frac{q-1}{2}} \binom{q}{2\ell}b^{2\ell}(-d)^\ell.$$
Reading this modulo $d$, we get $k\equiv \pm \pmod d$. This contradicts to (i). \\
When $a=\pm t$, then \eqref{eq1} becomes
$$k=\pm t^q\pm \sum_{\ell=1}^{\frac{q-1}{2}} \binom{q}{2\ell} t^{q-2\ell}b^{2\ell}(-d)^\ell.$$
As in previous case, reading this modulo $d$ one gets
$t^q\equiv \pm k\pmod d$. This contradicts to (ii).

We now consider the remaining case when $a=\pm k$. In this case, we have $k+m\sqrt{-d}=\left(\pm k+b\sqrt{-d}\right)^q$. We take the norm on both sides to get,
$$p^n=(k^2+b^2d)^{q}.$$
This implies
\begin{equation}\label{eq2}
db^2+k^2=p^{n/q}.
\end{equation}
Also we have,
\begin{equation}\label{eq3}
dm^2+k^2=p^n.
\end{equation}
Since $q$ is a prime divisor of $n$, so that \eqref{eq2} and \eqref{eq3} together show that $(|b|,n/q)$ and $(m,n)$ are two
distinct solutions of (\ref{de1}) in positive integers $x$ and $y$. This contradicts to Proposition \ref{PB}

We now consider the case when $d \equiv 3 \pmod 4$.
If $\alpha$ is an $q^{th}$ power of some element in $\mathcal{O}_{K_{k,p, n}}$,
then there are some rational integers $a,b$ with same parity such that
$$k+m\sqrt{-d}=\left( \frac{a+b\sqrt{-d}}{2} \right)^{q}.$$
If both $a$ and $b$ are even then one can proceed as in the case
$d \equiv 1,2 \pmod 4$, and gets a contradiction to Proposition \ref{PB}. Therefore one assumes that both $a$ and $b$ are odd.
Taking norm on both sides, one gets
\begin{equation}\label{eq4}
4p^{n/q}=a^2+b^2d.
\end{equation}

Since $a,b$ and $p$ are odd, so that reading (\ref{eq4}) modulo $8$, we obtain
$d \equiv 3 \pmod 8$. As $k+m\sqrt{-d}=\left( \frac{a+b\sqrt{d}}{2} \right)^q \in \mathbb{Z}[\sqrt{-d}]$ and $-d\equiv 5\pmod 8$, so that by Theorem \ref{PC} we obtain $q=3$.
Therefore,
$$k+m\sqrt{d}=\left( \frac{a+b\sqrt{-d}}{2} \right)^3.$$
We compare the real parts to get,
\begin{align}\label{eq5}
8k&=a(a^2-3b^2d).
\end{align}
Since $a$ is odd, so that $a=\pm 1,~\pm k'$ for some odd divisor $k'$ (other than $\pm 1$) of $k$.\\
If $a=1$ then \eqref{eq5} gives $8k=1-3b^2d$. This is not possible as $d$ and $k$ are positive integers.

For $a=-1$, \eqref{eq4} and \eqref{eq5} become,
$$
\begin{cases}
b^2d+1=4p^{n/3},\\
3b^2d-1=8k.
\end{cases}$$
This further implies
$$p^{n/3}=\frac{2k+1}{3},$$
which violates our assumption.
 
For $a=\pm k'$, \eqref{eq4} and \eqref{eq5} become,
$$
\begin{cases}
b^2d+k'^2=4p^{n/3},\\
-3b^2d+k'^2=\pm 8k/k'.
\end{cases}$$
This further implies
$$p^{n/3}=\frac{k'^3\pm 2k}{3k'}.$$
This again violates our assumption. Thus, we complete the proof.
\end{proof}

We are now in a position to present the proof of Theorem \ref{T1}.

\begin{proof}[\bf Proof of Theorem \ref{T1}]
Let $m$ be the positive integer such that $k^2-p^n=m^2d$ and $\alpha =k+m\sqrt{-d}$.
We note that $\gcd(\alpha,~\bar{\alpha})=1$ and $N(\alpha)=\alpha \bar{\alpha}=p^n$. Therefore, we can write $(\alpha)= \mathfrak{a}^n$ for some ideal $\mathfrak{a}$ in $ \mathcal{O}_{K_{k, p, n}}$. We assume that $\mathfrak{A}$ is the ideal class containing $\mathfrak{a}$ in the class group of $K_{k,p,n}$. If the order of $\mathfrak{A}$ is less than $n$, then we obtain an odd prime divisor $q$ of $n$ and an element $\beta \in \mathcal{O}_{K_{k, p,n}}$ satisfying $(\alpha)=(\beta)^{q}$. Since $d>3$ is square-free, so that the only units in $\mathcal{O}_{K_{k, p,n}}$ are $\pm1$, and thus these units can be absorbed into the $q$-th power of $\beta$. Hence we obatin $\alpha= \beta^q$ which contradicts to Proposition \ref{P2}. Therefore, the order of $\mathfrak{A}$ is $n$ which completes the proof.
\end{proof}

For an odd positive integer $n$ not divisible by $3$ and a fixed psoitive integer $k$, the conditions (i) and (ii) in Theorem \ref{T1} hold very often. This can be proved using Siegel's theorem on integral points on affine curves. More precisely, we prove the following result to show the infinitude of the imaginary quadratic fields of the form $K_{k, p, n}$ whose class group has an element of order $n$. We apply a celebrated theorem of Siegel \cite{LS} to prove the following result.
\begin{thm}\label{T2}
Let $n\geq 5$ be an odd integer not divisible by $3$. For each positive integer $k$, the class group of $K_{k, p, n}$ has an element of order $n$ for infinitely many odd primes $p$.
\end{thm}
\begin{remark} This is a generalization of \cite[Theorem 1.2]{CHKP} which can be obtained by considering $k$ as an odd prime in Theorem \ref{T2}. One can prove this theorem by a similar argument as in \cite[Theorem 1.2]{CHKP}. However for the sake of completeness, we give a proof of it.
\end{remark}
\begin{proof}
Let $n\geq 2$ be an odd integer not divisible by $3$, and $k$ an arbitrary positive integer. Let $p$ be an odd prime such that $\gcd(p, k)=1$. Then by Theorem \ref{T1}, the class group of $K_{k, p, n}$ has an element of order $n$ except for $t^q\equiv \pm k  \pmod {d}$, where $t (\ne k)$ is a divisor of $k$ and $q$ is a prime divisor of $n$. For $t^q\equiv \pm k \pmod {d}$, one gets $d\leq k+t^q$.
For any positive integer $D$, the curve
\begin{equation}\label{eq6}
Dx^2+p^2=y^n
\end{equation}
is an irreducible algebraic curve of genus bigger than $0$ (see \cite{WS}). Thus by Siegel's theorem (see \cite{ES, LS}), it follows that there are only finitely many integral points $(x,y)$ on the curve (\ref{eq6}). Therefore, for each positive integer $d$, there are at most finitely many primes $p$ such that
$$dm^2+k^2=p^n.$$
As $K_{k, p,n}=\mathbb{Q}(\sqrt{-d})$, it follows that there are infinitely many fields of the form $K_{k, p,n}$ for each positive integer $k$. Furthermore, if $p$ is sufficiently large, then for $k^2-p^n=-m^2d$, one gets $d>k+t^q$. Therefore by Theorem \ref{T1}, the class group of $K_{k, p, n}$ has an element of order $n$ for sufficiently large $p$.
\end{proof}
\section{Concluding result}
In Theorem \ref{T1}, we have proved that there exists an element of order $n$ in the class group of $\mathbb{Q}(\sqrt{k^2-p^n})$. By Theorem \ref{thm3.5}, it follows that if we consider the family $\mathbb{Q}(\sqrt{m^2q^2-\ell^n})$ then there exists a non-empty Zariski open set $U$ in the base of the family such that for all $b\in U$, the order of the $n$-torsion subgroup of $\CH^1(S_{\mathbb{Z},b})$ is invariant. Not only that, these $n$-torsion subgroups in the class groups of the family of imaginary quadratic fields mentioned above, varies in a family.

\medskip

\noindent\textbf{Acknowledgements.}
A part of this paper was completed while A. Hoque was visiting Professor Jianya Liu at Shandong University. He is thankful to Professor Liu and the University for hosting him during this project. The authors would also like to thank Professor K. Chakraborty for his valuable suggestions to improve the presentation of this paper.  The authors are grateful to the anonymous referee for a careful reading of the paper and for suggesting several improvements. This work is supported by SERB MATRICS grant (No. MTR/2021/000762) and SERB  CRG grant (No. CRG/2023/007323), Govt. of India. 

\subsection*{Declaration} The authors declare that there are no conflict of interests.

\subsection*{Data Availability Statement} This manuscript has no associate data.


\begin{thebibliography}{99}
\bibitem{AP} A. Agboola and G. Pappas, {\it Line bundles, rational points and ideal classes}, Math. Res. Lett. {\bf 7} (2000), no. 5-6, 709--717.

\bibitem{AC55} N. C. Ankeny and S. Chowla, {\it On the divisibility of the class number of quadratic fields}, Pacific J. Math. {\bf 5} (1955), 321--324.

\bibitem{BG} K. Banerjee and V. Guletskii, {\it Etale monodromy and rational equivalence for one cycles on cubic hypersurfaces in $P^5$},  Sbornik Math.  {\bf 211}  (2020), no. 2, 161--200. 

\bibitem{BS01} Y. Bugeaud and T. N. Shorey, {\it On the number of solutions of the generalized Ramanujan-Nagell equation}, J. Reine Angew. Math. {\bf 539} (2001), 55--74.

\bibitem{CHKP} K. Chakraborty, A. Hoque, Y. Kishi and P. P. Pandey, {\it Divisibility of the class numbers of imaginary quadratic fields}, J. Number Theory {\bf 185} (2018), 339--348.


\bibitem{CH-19} K. Chakraborty and A. Hoque, {\it Exponents of class groups of certain imaginary quadratic fields}, Czechoslovak Math. J. {\bf 70} (2020), no. 4, 1167--1178.

\bibitem{CH-23} K. Chakraborty and A. Hoque, {\it Lehmer sequence approach to the divisibility of class numbers of imaginary quadratic fields}, Ramanujan J. {\bf 60} (2023), no. 4, 913--923.

\bibitem{Cohn1} J. H. E. Cohn, {\it Square Fibonacci numbers, etc.}, Fibonacci Quart. {\bf 2} (1964), 109--113.


\bibitem{ES} J. -H. Evertse and J. H. Silverman, {\it Uniform bounds for the number of solutions to $Y^n=f(X)$}, Math. Proc. Camb. Phil. Soc. {\bf 100} (1986), 237--248.
\bibitem{FGA} B. Fantechi, L. Gottsche, L. Illusie, S. Kleiman, N. Nitsure and  A. Vistoli, {\it Fundamental Algebraic Geometry: Grothendieck's FGA Explained}, Mathematical surveys and monographs, AMS, Vol. {\bf 123}, 2005.
\bibitem{GS} H. Gillet and C. Soule, {\it Arithmetic intersection theory}, Publications Math\'ematiques de l'IH\'ES {\bf 72} (1990), 93--174.

\bibitem{GL} J. Gillibert and A. Levin, {\it Pulling back torsion line bundles to ideal classes}, Math. Res. Lett. {\bf 19} (2012), no. 6, 1171--1184.

\bibitem{GL-18} J. Gillibert and A. Levin, {\it A geometric approach to large class groups: A survey}, In: Chakraborty, K., Hoque, A., Pandey, P. (eds) Class Groups of Number Fields and Related Topics, 1-- 15, Springer, Singapore, 2020.

\bibitem{GI} J. Gillibert, {\it From Picard groups of hyperelliptic curves to  class groups of quadratic fields}, Trans. Amer. Math. Soc. {\bf 374} (2021), 3919--3946. 

\bibitem{GR01} B. H. Gross and D. E. Rohrlich, {\it Some results on the Mordell-Weil group of the Jacobian of the Fermat curve}, Invent. Math. {\bf 44} (1978), 201--224.

\bibitem{HC-18} A. Hoque and K. Chakraborty, {\it Divisibility of class numbers of certain families of quadratic fields}, J.  Ramanujan Math. Soc. {\bf 34} (2019), no. 3, 281--289.

\bibitem{HO20} A. Hoque, {\it On the exponents of class groups of some families of imaginary quadratic fields}, Mediterr. J. Math. {\bf 18} (2021), no. 4, Paper no. 153, 13 pp.

\bibitem{H-22} A. Hoque, {\it On a conjecture of Iizuka}, J. Number Theory {\bf 238} (2022), 464--473. 

\bibitem{KI09} Y. Kishi, {\it Note on the divisibility of the class number of certain imaginary quadratic fields}, Glasgow Math. J. {\bf 51} (2009), 187--191;
corrigendum, ibid. {\bf 52} (2010), 207--208.

\bibitem{Ko} J. Kollar, {\it Rational curves on algebraic varieties}, Springer, New York, 1996.

\bibitem{LJ43} W. Ljunggren, {\it Some theorems on indeterminate equations of the form $\frac{x^n-1}{x-1}=y^q$}, Norsk Mat. Tidsskr. {\bf 25} (1943), 17--20.

\bibitem{LO09} S. R. Louboutin, {\it On the divisibility of the class number of imaginary quadratic number fields}, Proc. Amer. Math. Soc. {\bf 137} (2009), 4025--4028.

\bibitem{M} D. Mumford, {\it Rational equivalence for $0$-cycles on surfaces}, J. Math Kyoto Univ. {\bf 9} (1968), 195--204.



\bibitem{R} A. A. Ro\u{i}tman, {\it $\Gamma$-equivalence of zero-dimensional cycles}, Mat. Sb. (N. S.) {\bf 86} (1971), 557--570.

\bibitem{Ry} D. Rydh, {\it Families of cycles and Chow schemes}, PhD Thesis, KTH, Stockholm Sweden 2008.

\bibitem{WS} W. M. Schmidt, Equations over finite fields: an elementary approach,
Lecture Notes in Mathematics {\bf 536}, Springer-Verlag, Berlin-New York, 1976.

\bibitem{LS} C. L. Siegel, {\it \"Uber einige Anwendungen Diophantischer Approximationen}, Abh. Preuss. Akad. Wiss. Phys. Math. Kl. {\bf 1} (1929), 1-70; Ges. Abh., Band {\bf 1}, 209--266.

\bibitem{SO00} K. Soundararajan, {\it Divisibility of class numbers of imaginary quadratic fields}, J. London Math. Soc. {\bf 61} (2000), 681--690.

\bibitem{SV} A. Suslin and V. Voevodsky, {\it Relative cycles and Chow sheaves}, Cycles, transfers, motivic homology theories, 10-86, Annals of Math studies.

\bibitem{stacks-project} The Stacks project authors, {\it The Stacks project}, \url{https://stacks.math.columbia.edu}, 2023. 

\bibitem{Vo} C. Voisin, {\it Hodge theory and complex algebraic geometry. II}, Cambridge Studies in Advanced Mathematics {\bf 77}, Cambridge University Press, Cambridge, 2003.

\bibitem{Voi} C. Voisin, {\it Unirational threefolds with no universal codimension 2 cycle}, Invent. Math. {\bf 201} (2015), no. 1, 207--237.


\end{thebibliography}
\end{document}